\documentclass[11pt,english]{article}

\pdfoutput=1

\usepackage[margin = 1.1in]{geometry}

\usepackage{amsthm}
\usepackage{amsmath}
\usepackage{amssymb}
\usepackage{setspace}
\usepackage{comment}
\usepackage{mathtools}
\usepackage{graphicx}
\usepackage[hidelinks]{hyperref}
\usepackage{thm-restate}
\usepackage{cleveref}

\usepackage{enumerate}
\usepackage{framed}
\usepackage{subcaption}

\usepackage{floatrow}

\theoremstyle{plain}

\newtheorem*{thm*}{Theorem}
\newtheorem{thm}{Theorem}[section]
\Crefname{thm}{Theorem}{Theorems}

\newtheorem*{lem*}{Lemma}
\newtheorem{lem}[thm]{Lemma}
\Crefname{lem}{Lemma}{Lemmas}

\newtheorem*{claim*}{Claim}
\newtheorem{claim}[thm]{Claim}
\crefname{claim}{Claim}{Claims}
\Crefname{claim}{Claim}{Claims}

\Crefname{prop}{Proposition}{Propositions}

\newtheorem{cor}[thm]{Corollary}
\Crefname{cor}{Corollary}{Corollaries}

\newtheorem{conj}[thm]{Conjecture}
\Crefname{conj}{Conjecture}{Conjectures}

\Crefname{qn}{Question}{Questions}

\newtheorem{obs}[thm]{Observation}
\Crefname{obs}{Observation}{Observations}

\theoremstyle{definition}
\newtheorem{prob}[thm]{Problem}
\Crefname{prob}{Problem}{Problems}

\Crefname{defn}{Definition}{Definitions}

\theoremstyle{remark}

\renewenvironment{proof}[1][]{\begin{trivlist}
\item[\hspace{\labelsep}{\bf\noindent Proof#1.\/}] }{\qed\end{trivlist}}

\newcommand{\remove}[1]{}

\newcommand{\mQ}{\mathcal{Q}}
\newcommand{\mM}{\mathcal{M}}

\newcommand{\mF}{\mathcal{F}}
\newcommand{\mK}{\mathcal{K}}

\newcommand{\mL}{\mathcal{L}}
\newcommand{\mP}{\mathcal{P}}

\newcommand{\mS}{\mathcal{S}}

\begin{document}


\title{Highly linked tournaments with large minimum out-degree}
\date{\vspace{-5ex}}
\author{
    Ant{\'o}nio Gir\~{a}o
    \thanks{
        Department of Pure Mathematics and Mathematical Statistics, 	
        University of Cambridge, 
       	Wilberforce Road, 
        CB3\;0WB Cambridge, 
        UK;
        e-mail:
        \mbox{\texttt{A.Girao@dpmms.cam.ac.uk}}\,.
    } 
    \and
     Richard Snyder
     \thanks{
     	Department of Mathematical Sciences, The University
	of Memphis, Memphis, Tennessee;
	e-mail: \mbox{\texttt{rjsnyder23@gmail.com}}\,
	}
}

\maketitle
\singlespace

\begin{abstract}

    \setlength{\parskip}{\medskipamount}
    \setlength{\parindent}{0pt}
    \noindent
    
    We prove that there exists a function $f:\mathbb{N} \rightarrow \mathbb{N}$ such 
    that for any positive integer $k$, if $T$ is a strongly $4k$-connected tournament with 
    minimum out-degree at least $f(k)$, then $T$ is $k$-linked. This resolves a
     conjecture of Pokrovskiy up to a factor of $2$ of the required connectivity. Along the way, we show that a tournament
    with sufficiently large minimum out-degree contains a subdivision of a complete directed 
    graph. This result may be of independent interest.

  \end{abstract}

\section{Introduction}

Given a positive integer $k$, a graph is said to be $k$-\emph{linked} if for any two disjoint sets of vertices $\{x_1, \ldots , x_k\}$ and $\{y_1, \ldots , y_k\}$ there are vertex disjoint paths $P_1, \ldots , P_k$ such that $P_i$ joins $x_i$ to $y_i$ for $i = 1, \ldots, k$. Clearly, $k$-linkedness is a stronger notion than $k$-connectivity. But how much stronger is it? Larman and Mani~\cite{Larman_Mani} and Jung~\cite{Jung} showed that there is an $f(k)$ such that any $f(k)$-connected graph is $k$-linked. They based their result on a theorem of Mader~\cite{Mader1}, which
implies that for any $k$, any sufficiently connected graph contains a subdivision of a complete graph on $3k$ vertices, and noticed that any $2k$-connected graph containing such a subdivision must be 
$k$-linked. Their proofs show that $f(k)$ can be taken to be exponential in $k$. 
Later, Bollob{\'a}s and Thomason~\cite{BollobasThomason} proved that $f(k) = 22k$ will do. 

The definitions of $k$-connectivity and $k$-linkedness carry over to directed graphs. A directed graph is \emph{strongly connected} if for any pair of distinct vertices $x$ and $y$ there is a directed path from $x$ to $y$, and is strongly $k$-connected if it remains connected upon removal of any set of at most $k-1$ vertices. In the sequel, we shall omit the use of the word
`strongly' with the understanding that we always mean strong connectivity.  A directed graph $D$ is $k$-\emph{linked} if 
for any two disjoint sets of vertices $\{x_1, \ldots , x_k\}$ and $\{y_1, \ldots, y_k\}$ there are pairwise vertex disjoint directed 
paths $P_1, \ldots , P_k$ such that $P_i$ has initial vertex $x_i$ and terminal vertex $y_i$ for every $i \in [k]$. Thus, $D$ 
is $1$-linked if and only if it is connected.


Directed graphs exhibit quite different behaviour from undirected graphs with respect to the relations they bear between connectivity and linkedness. 
Indeed, Thomassen~\cite{Thomassen_construct} constructed directed graphs with arbitrarily large connectivity which are not even $2$-linked.
 Since large connectivity does not necessarily imply linkedness for general directed graphs, it is natural to consider the situation for a restricted class of directed graphs, namely, \emph{tournaments}. Thomassen~\cite{Thomassen_tourn} proved that there is a $g(k)$ such that every $g(k)$-connected tournament is $k$-linked, 
with $g(k) =C k!$, for some absolute constant $C$.
Greatly improving Thomassen's bound on $g(k)$, K{\"u}hn, Lapinskas, Osthus, and Patel~\cite{Lap_Kuhn_Osthus} showed that one may 
take $g(k) = 10^4k\log k$ and still ensure $k$-linkedness. They went on to conjecture that $g(k)$ may be taken to be linear in $k$. 
Pokrovskiy~\cite{Pokrovskiy} resolved this conjecture by showing that any $452k$-connected tournament is $k$-linked. Except for small values of $k$, an optimal bound for $g(k)$ is not known. 
Bang-Jensen~\cite{Bang-Jensen} showed that any $5$-connected tournament is $2$-linked, and there exists a family of $4$-connected tournaments which are 
not $2$-linked. Moreover, it is easy to construct $(2k-2)$-connected tournaments with arbitrarily large out and in-degree which are not $k$-linked:  consider the blow up of a directed triangle with vertex sets $A, B, C$ such that $|C| = 2k - 2$ 
and $A$ and $B$ have size at least $2k$.

Going back to undirected graphs for a moment, if some density conditions are assumed on the graph, then Bollob{\'a}s and Thomason's $22k$ can be taken all the way down to $2k$, since Mader~\cite{Mader1} proved that a graph with sufficiently large average degree contains a subdivision of a complete graph of order $3k$. Note that $2k$ is close to the theoretical minimum connectivity in any $k$-linked graph (a $k$-linked graph is necessarily $(2k-1)$-connected).
Recently, Thomas and Wollan~\cite{ThomasWollan2} showed that any $2k$-connected graph with average degree at least $10k$ is $k$-linked, greatly reducing the bound on the required average degree. Motivated by this result, Pokrovskiy~\cite{Pokrovskiy} conjectured that a similar phenomenon should occur for tournaments with a natural 
`density' condition:  high minimum out-degree and in-degree. In particular, he conjectured that there is a function $f: \mathbb{N} \rightarrow \mathbb{N}$ such that any $2k$-connected tournament with minimum out and in-degree at least $f(k)$ is $k$-linked.
Here is our main result, which makes progress on this conjecture. 

\begin{thm}\label{thm:main}
For every positive integer $k$ there exists $f(k)$ such that 
every $4k$-connected tournament $T$ with $\delta^+(T) \ge f(k)$
is $k$-linked.
\end{thm}

Note that we do not need to assume any lower bound on the minimum in-degree. We also remark that our proof yields the upper bound $f(k) \leq 2^{2^{Ck^4}}$ for some constant $C$, which comes from the upper bound on $d(k)$ in the next theorem. 

Recall that the complete directed graph $\overrightarrow{K}_k$ is the directed graph on $k$ vertices 
where, for every pair $x, y$ of distinct vertices, both $xy$ and $yx$ are present.  
In order to prove \Cref{thm:main} we shall show that large minimum out-degree allows 
us to embed subdivisions of the complete directed graph $\overrightarrow{K}_k$. As we mentioned 
earlier, Mader~\cite{Mader1} showed 
that for any positive integer $k$ there is $g(k)$ such that 
any graph with average degree at least $g(k)$ contains a subdivision of $K_k$.
The following theorem can be viewed as an analogue of Mader's result for tournaments, replacing 
`average degree' with `minimum out-degree', and may be of independent interest.

\begin{restatable}{thm}{thmSubdivision}\label{thm:subdivision}
For any positive integer $k$ there exists a $d(k)$ such that 
the following holds. If $T$ is a tournament with $\delta^+(T) \ge d(k)$, then
$T$ contains a subdivision of $\overrightarrow{K}_k$.
\end{restatable}

Our proof gives the bound $d(k)\leq 2^{2^{Ck^2}}$ for some constant $C$. We leave the determination of the smallest $d(k)$ such that \Cref{thm:subdivision} holds as an open problem.
We also note that, as shown by Mader~\cite{Mader3}, this theorem does not hold if we replace $T$ by a general digraph. This fact also follows from a result 
of Thomassen~\cite{Thomasseneven}, who showed that for every integer $n$ there exist digraphs on $n$ vertices with minimum out-degree at least $\frac{1}{2}\log n$ which do not contain a directed cycle of even length. But since any 
subdivision of $\overrightarrow{K}_3$ must contain an even directed cycle, these digraphs do not contain any subdivision of a complete directed 
graph.

In order to prove \Cref{thm:main}, we shall need a little more than \Cref{thm:subdivision}. Roughly speaking, we shall first embed in $T$ a subdivided 
$\overrightarrow{K}_k$, and then attach a few additional paths to it (see \Cref{sec:subdivision}).

\subsection{Organization and Notation}
The remainder of this paper is organized as follows.
In \Cref{sec:subdivision}, we prove \Cref{thm:subdivision} which allows us to embed subdivisions
of a complete directed graph and related structures in tournaments with high minimum out-degree. In \Cref{sec:main}, we shall prove one 
preparatory lemma and then finish our proof of \Cref{thm:main}. Our final section concludes with some open problems.

Our notation is standard. Thus, for a directed graph $D$ we use $N^+(x), N^-(x), d^+(x), d^-(x)$ to denote the out-neighbourhood, in-neighbourhood, 
out-degree, and in-degree of a vertex $x$, respectively. We use $\delta^+(D)$ to denote the minimum out-degree of $D$. A directed path 
$P = x_1\ldots x_\ell$ in $D$ is a sequence of distinct vertices such that $x_{i}x_{i+1}$ is an edge for every $i =1, \ldots , \ell -1$. We call 
$x_1$ the \emph{initial vertex} and $x_\ell$ the \emph{terminal vertex} of $P$. The length of $P$ 
is the number of its directed edges. We say that $P$ is \emph{internally disjoint} from some subset $X \subset V(D)$ if 
$\ell \ge 3$ and $\{x_2, \ldots , x_{\ell -1}\} \cap X = \varnothing$. If $A$ and $B$ are subsets of $V(D)$, then we shall write $A \rightarrow B$ if every edge with one endpoint in $A$ and the 
other endpoint in $B$ is directed from $A$ to $B$.
Lastly, if $\mP$ is a family of directed paths in a digraph, then we use $\bigcup \mP$ to denote the set $\bigcup_{P \in \mP} V(P)$.

\section{Proof of \Cref{thm:subdivision}}\label{sec:subdivision}

The first proofs of the result that graphs with sufficiently large connectivity are $k$-linked use a result of Mader, which allows one 
to embed a subdivision of a complete graph in a graph with sufficiently large average degree. Our proof of \Cref{thm:main} follows 
a similar strategy. In order to proceed, we need a directed analogue of Mader's result for tournaments: we prove this in the 
present section.
We shall use the following simple lemma of Lichiardopol~\cite{Lichiardopol} (independently rediscovered by Havet and Lidick{\'y}~\cite{Havet-Lidicky}). 
We include a short proof for convenience of the reader.
\begin{lem}\label{lem:Hav_Lid}
Every tournament with minimum out-degree at least $k$ has a subtournament
with minimum out-degree $k$ and order at most $3k^2$.
\end{lem}
\begin{proof}
Let $T$ be a tournament with minimum out-degree at least $k$, and let $T'$ be a vertex-minimal subtournament of $T$ 
such that $\delta^+(T') \ge k$. Denote by $L$ the collection of vertices in $T'$ with out-degree $k$ in $T'$, and let $|T'| = t$ and $|L| = \ell$.
By minimality, for every vertex $v \in T'$ we have $\delta^+(T'\setminus \{v\}) \le k - 1$. Hence, every vertex in $T' \setminus L$ has an 
in-neighbour in $L$, and so there are at least $t - \ell$ edges from $L$ to $T' \setminus L$. On the other hand, the number of such edges 
is exactly
\[
	\ell k - \binom{\ell}{2},
\]
and so $t - \ell \le \ell k - \ell^2/2 + \ell/2$. It follows that
\[
	\ell^2 - \ell(2k+3) + 2t \le 0,
\]
implying the bound $(2k+3)^2 - 8t \ge 0$. In other words, $t \le \frac{1}{8}(2k+3)^2$, so since $t$ must be an integer we get
$t \le \frac{1}{8}((2k+3)^2 - 1) = k^2/2 + 3k/2 + 1 \le 3k^2$, as required.

\end{proof}

We are now ready to prove \Cref{thm:subdivision}. In the following, for a positive integer $k$ and nonnegative integer $m \le 2\binom{k}{2}$, an $m$-\emph{partial} $\overrightarrow{K}_k$ is any spanning subdigraph
of $\overrightarrow{K}_k$ with precisely $m$ directed edges present. Our proof shows that we can find a subdivision 
of $\overrightarrow{K}_k$ by inductively finding subdivisions of $m$-partial $\overrightarrow{K}_k$'s for each $m \le 2\binom{k}{2}$. 

\begin{proof}[ of \Cref{thm:subdivision}]
For a positive integer $k$ and nonnegative integer $m \le 2\binom{k}{2}$, let $d(k, m)$ denote the smallest positive integer such that any tournament
with $\delta^+(T) \ge d(k, m)$ contains a subdivision of an $m$-partial complete directed graph on $k$ vertices.
We shall show that if $m < 2\binom{k}{2}$, then $d(k, m+1) \le 7d(k, m)^2$. We use induction on $k$, and for each fixed $k$, induction on $m$.
For $k = 1$ there is nothing to show and we can take $d(1, 0) = 1$. So let us assume $k \ge 2$ is given and that we can embed a subdivision of an $m$-partial $\overrightarrow{K}_k$ in any tournament with minimum out-degree at least $d(k, m)$, and let $T$ be a tournament with $\delta^+(T) \ge 7d(k, m)^2$.

\begin{claim}\label{claim:out_nbhd}
We may assume that there is a subdivision of an $m$-partial $\overrightarrow{K}_k$ contained in the out-neighbourhood 
of some vertex of $T$, and which spans at most $3d(k, m)^2$ vertices.
\end{claim}
\begin{proof}
Since certainly we have $\delta^+(T) \ge d(k, m)$, by \Cref{lem:Hav_Lid} we may find a subtournament $T'$ of size at most $3d(k, m)^2$ and with
minimum out-degree at least $d(k, m)$. By induction we may embed in $T'$ a subdivision of an $m$-partial $\overrightarrow{K}_k$. Denote this subdivision by $K$. We wish to add a missing directed edge, say $xy$. In other words, we must find a directed path from $x$ to $y$ in $T$
such this path is internally disjoint from $V(K)$. Let $T'' = T\setminus T'$ and partition it into strongly connected subtournaments 
$T'' = S_1 \cup \cdots \cup S_\ell$ such that $S_i \rightarrow S_{j}$ for all $1 \le i < j \le \ell$ (unless, of course, $T''$ itself is strongly connected). Observe that since $d^+(x) \ge 7d(k, m)^2$ and
$|T'| \le 3d(k, m)^2$, we have that $x$ has an out-neighbour in $T''$. Therefore, if some vertex of $S_\ell$ is joined to $y$ we are done, as 
we can find a directed path from $x$ to $y$ outside of $T'$. So we may assume that $S_\ell \subseteq N^+(y)$.
Now, as $|T'| \le 3d(k, m)^2$ and no vertex of $S_\ell$ is joined to any vertex of $S_i$ for $i < \ell$, we have that 
\[
	\delta^+(S_\ell) \ge 7d(k, m)^2 - 3d(k, m)^2 \ge d(k, m),
\]
Applying \Cref{lem:Hav_Lid} to $S_\ell$, we find a subtournament 
$S \subseteq S_\ell$ such that $\delta^+(S) \ge d(k, m)$ and with size at most $3d(k, m)^2$. It follows by induction that we may embed a subdivision of
an $m$-partial $\overrightarrow{K}_k$ in $S$. But since $S \subseteq S_\ell \subseteq N^+(y)$ and $|S|  \le 3d(k, m)^2$, the claim holds.
\end{proof}

By \Cref{claim:out_nbhd}, choose a vertex $z$ with the smallest possible minimum out-degree satisfying the property that 
there is a subdivision of an $m$-partial $\overrightarrow{K}_k$ contained in $N^+(z)$ spanning at most 
$3d(k, m)^2$ vertices. Denote by $N$ the out-neighbourhood of $z$ and $K_z$ the subdivision with $K_z \subseteq N$.
We wish to add one more directed edge to this subdivision, say $uv$ with $u, v \in K_z$. From $N$ remove all vertices 
of $K_z$ except for $u$ and $v$ and call this set $N'$. If $T[N']$ is strongly connected then we are done; otherwise, 
partition $T[N']$ into strongly connected subtournaments, say $T[N'] = S'_1 \cup \cdots \cup S'_t$ where $S'_i \rightarrow S'_j$ 
for all $1 \le i < j \le t$. Suppose that some vertex $w \in S'_t$ is joined to a vertex $w' \in N^-(z)$. Then since there is a directed 
path $P$ from $u$ to $w$ in $T[N']$ we have that $uPww'zv$ is a directed path from $u$ to $v$ which avoids $K_z \setminus \{u, v\}$. Hence we may assume that every vertex of $N^-(z)$ dominates $S'_t$. But then, since $|K_z| \le 3d(k, m)^2$ and there are no edges from $S'_t$ to $S'_i$ for $i < t$, one has that $\delta^+(S'_t) \ge 7d(k, m)^2 - 3d(k, m)^2 = 4d(k, m)^2$. So we can repeat the argument in \Cref{claim:out_nbhd} to $S_t'$ with 
minimum out-degree $4d(k, m)^2$ instead of $7d(k, m)^2$ (observe that we need $4d(k,m)^2 - 3d(k, m)^2 \ge d(k, m)$ to hold, which is 
clearly true). Accordingly, there is a vertex $q \in S'_t$ such that $N^+(q)$ contains a subdivision of an $m$-partial $\overrightarrow{K}_k$ 
spanning at most $3d(k, m)^2$ vertices. However, since $\bigcup_{i < t}S'_i \neq \varnothing$ (as $T[N']$ is not strongly connected), and $q$ 
is not joined to any vertex of $\bigcup_{i <t}S'_i \cup N^-(z)$, we have $d^+(q) < d^+(z)$, a contradiction to the minimality of $z$. This 
completes the proof of \Cref{thm:subdivision}, as we may take $d(k) = d(k, 2\binom{k}{2})$.

\end{proof}

We now need to embed a slightly more complicated structure in $T$. In particular, we shall need to attach a few special paths to our 
subdivided complete directed graph. Say a subdivision $\mS$ is \emph{minimal} in a tournament $T$ if all of its paths have minimal length. This implies that every path in $\mS$ is \emph{backwards transitive}:  if $x_1\ldots x_t$ is a path in $\mS$ between branch vertices, then $x_ix_j \notin E(T)$ 
whenever $i \in [t-2]$ and $i < j+1$. Let $\mK_r^{\min}$ denote a minimal subdivision of a $\overrightarrow{K}_r$. Since any subdivision of $\overrightarrow{K}_r$ contains a minimal subdivision, \Cref{thm:subdivision} allows us to find a $\mK_r^{\min}$ in tournaments with 
sufficiently large out-degree. If $U$ denotes the set of branch vertices of this subdivision, then for every $u, v \in U$, 
$\mK_r^{\min}$ consists of directed paths $P_{uv}, P_{vu}$ going from $u$ to $v$ and from $v$ to $u$, respectively. Since $T$ is a tournament and $\mK_r^{\min}$ is minimal, precisely one of these paths is a directed edge.

Now we define our augmented subdivision, denoted by $\mK_r^*$, as follows.
Let $\mK$ denote a copy of $\mK_r^{\min}$ in $T$. The \emph{branch vertices} of $\mK_r^*$ are precisely the branch vertices of $\mK$; denote this set 
by $U$. We form $\mK_r^*$ by adding a collection $\mathcal{L}$ of special `loop' paths in the following manner. For each pair $u, v \in U$,
 if, say, $P_{uv}$ is the path between $u$ and $v$ in $\mK$ of length at least two, then each of $u$ and $v$ 
 has an associated directed path from $\mathcal{L}$: one directed path $L_{uv}^u$ going from the second vertex of $P_{uv}$ to $u$, and another directed path $L_{uv}^v$ going from $v$ to the penultimate vertex of $P_{uv}$; we require that these paths are internally disjoint from $V(\mK)$. We also impose that the paths in $\mL$ are minimal and hence backwards transitive. 
 For $u \in U$, we let $\mathcal{L}_u$ denote the collection of paths in $\mathcal{L}$ which contain $u$.
Note that $\mK_r^*$ and $\mK_r^{\min}$ really denote \emph{families} of subdigraphs which depend on the underlying tournament $T$. When we speak of `a $\mK_r^*$' we really mean `a member of $\mK_r^*$ in $T$'; we hope this usage of notation does not cause confusion, but we think 
that it is simpler.
Now the proof of the existence of a $\mK_{r}^{*}$ follows exactly in the same way as the proof of  \Cref{thm:subdivision}, namely by induction on the number of `loops'. We state it as a corollary and provide only a sketch of the proof.

\begin{cor}\label{cor:subdivision2}
For any positive integer $k$ there exists a $d^{*}(k)$ such that 
the following holds. If $T$ is a tournament with $\delta^+(T) \ge d^*(k)$, then
$T$ contains a $\mK^*_k$.
\end{cor}

\begin{proof}[ (Sketch)]

Similarly as in \Cref{thm:subdivision}, for a positive integer $k$ and nonnegative integer $m \le 2\binom{k}{2}$, an $m$-\emph{partial} $\mK^*_k$ is any minimal subdivision of $\overrightarrow{K}_k$ with precisely $m$ loop paths present. Let 
 $d^*(k, m)$ denote the smallest positive integer such that any tournament
with $\delta^+(T) \ge d^*(k, m)$ contains a subdivision of an $m$-partial $\mK^*_k$.
We show, as before, that if $m < 2\binom{k}{2}$, then $d^*(k, m+1) \le 7d^*(k, m)^2$.
For $k = 1$ there is nothing to show and we can take $d^*(1, 0) = 1$. So assume $k \ge 2$ is given. Then $d^*(2, 0)$ exists by \Cref{thm:subdivision} (i.e., we can embed a subdivision of $\overrightarrow{K}_2$ which contains a minimal such subdivision).  Thus let $m \ge 1$ and suppose we can embed an $m$-partial $\mK^*_k$ in any tournament with minimum out-degree at least $d^*(k, m)$. Let $T$ be a tournament with $\delta^+(T) \ge 7d^*(k, m)^2$. Then the same proof used to show \Cref{thm:subdivision} gives that we may attach one more loop path, which we may assume has minimal length. Therefore we can embed an $(m+1)$-partial $\mK_k^*$ in $T$, as claimed.
\end{proof}

\section{Proof of the main theorem}\label{sec:main}

In this section we finish the proof of \Cref{thm:main}. The structure of the proof is as follows. First, assuming the 
minimum degree of our tournament is sufficiently large, we shall embed in $T$ a copy $\mS$ of $\mK_{r}^*$ where $r = r(k)$ is 
sufficiently large. If $x_1, \ldots, x_k$, $y_1, \ldots , y_k$ are the vertices we want to link, then we shall show that there exists
a collection of $k$ directed paths going from the $x_i$'s to the branch vertices of $\mS$, and a collection of $k$ directed paths going from the branch vertices 
of $\mS$ to the $y_i$'s, all of these paths being pairwise vertex disjoint. Here we only use the assumption that $T$ is $4k$-connected (see \Cref{lem:reversing_paths} below). Finally, we show that, provided one chooses these paths appropriately, one can link each $x_i$ to $y_i$ 
by rerouting the paths through $\mS$. The rerouting step is more complicated than one might expect, and we shall see that we do 
need the slightly richer structure $\mK_r^*$ rather than just a subdivided complete directed graph.

We need a small bit of terminology first before proceeding.
If $X$ and $Y$ are two disjoint sets of vertices in a directed graph, then
we say that there is an \emph{out-matching} (resp., \emph{in-matching}) of $X$ to $Y$ if there is a matching from 
$X$ into $Y$ such that all matching edges are directed from $X$ to $Y$ (resp., directed from $Y$ to $X$).

\begin{lem}\label{lem:reversing_paths}
Let $T$ be a $4k$-connected tournament. Suppose $A, B \subset V(T)$ are two 
disjoint subsets of size $k$, and let $L \subset V(T)$ be a set of $4k$ vertices disjoint from $A \cup B$.
Then there are $k$ directed 
paths from $A$ to $L$, and $k$ directed paths from $L$ to $B$, all these paths pairwise vertex 
disjoint and internally disjoint from $L$.
\end{lem}

\begin{proof}

Choose two disjoint subsets $W_A, W_B$ disjoint from $A \cup B\cup L$ with maximum size 
subject to the following properties:
\begin{itemize}
\item Every vertex in $W_A$ has at least $2k$ out-neighbours in $L$, and every vertex in
	$W_B$ has at least $2k$ in-neighbours in $L$.
	
\item There is an in-matching $\mM_A$ from $W_A$ to $A$, and an out-matching $\mM_B$ from $W_B$ to $B$.

\end{itemize}
We shall assume, without loss of generality, that $|W_A| \le |W_B|$. Let $A'$ denote the set of $|W_A|$ vertices in $A$
that are incident with an edge of $\mM_A$, and let $A'' = A \setminus A'$. Let $B', B''$ denote the analogous sets of vertices in $B$.
 As $T$ is $4k$-connected, we can find 
pairwise vertex disjoint directed paths from some $k - |W_B|$ vertices of $L$ to $B''$ avoiding $A \cup W_A \cup B' \cup W_B$. 
Choose a collection of 
such paths $\mP$ which minimizes $|\bigcup \mP|$, and subject to that, maximizes the number of paths whose
second vertex has at least $2k$ in-neighbours in $L$. Partition $\mP$ into sets $\mP', \mP''$ where the former 
denotes the collection of paths in $\mP$ whose second vertex has at least $2k$ in-neighbours in $L$, and the latter denotes the collection of remaining 
paths. Denote by $X'$ the set of all second and third vertices on paths in $\mP'$, and denote by $X''$ the set of all first and second vertices 
on paths in $\mP''$. Consider the set $Y:= A' \cup W_A \cup X' \cup X'' \cup B \cup W_B$ and note that we can bound the size of $Y$ as
\[
	|Y| \le 2|W_A| + 3(k - |W_B|) + 2|W_B|.
\]
We now find $k - |W_A|$ disjoint directed paths from the vertices in $A''$ to some subset of $L$, avoiding $Y$.
This is possible since $T$ is $4k$-connected and 
\begin{align*}
4k - |Y| &\ge 4k - (2|W_A| + 3(k - |W_B|) + 2|W_B|) \\
&= k - 2|W_A| + |W_B| \ge k - |W_A|,
\end{align*}
where the last inequality holds since we are assuming that $|W_A| \le |W_B|$. Therefore, choose a collection $\mQ$ of pairwise disjoint directed paths 
from $A''$ to $L$ avoiding $Y$ with $|\bigcup \mQ|$ as small as possible.
We claim that these new paths do not intersect any path from $\mP$:
\begin{claim}\label{claim:paths_disjoint}
No path from $\mQ$ intersects a path from $\mP$.
\end{claim}
\begin{proof}
Suppose that some path $Q \in \mQ$ intersects a path $P \in \mP$. Let $P = x_1 \ldots x_s$ and $Q = y_1\ldots y_t$, 
and let $L_A = (\bigcup \mQ) \cap L$ and similarly $L_B = (\bigcup \mP) \cap L$. 
We consider two cases, 
according to whether $P \in \mP'$ or $P \in \mP''$. Suppose first the former holds, and let $y_i$ ($i\ge 2$) be the 
first vertex of $Q$ that intersects $P$. We may assume that $y_i \neq x_1$; indeed, if $y_i = x_1$, then $|L_A \cup L_B| \le 2k -1$, 
and since $P \in \mP'$, we have that $x_2$ has at least $2k$ in-neighbours in $L$. Therefore, we may choose some in-neighbour 
$x'$ disjoint from $L_A \cup L_B$ and replace $P$ with $P' := x'x_2\ldots x_s$. Moreover, since the paths in $\mQ$ 
avoid $\{x_2, x_3\}$, we may assume that $y_i = x_j$, for some $j \ge 4$. Consider $y_{i-1}$ and pick any vertex $z \in L \setminus (L_A \cup L_B)$. 
If $y_{i-1}z \in E(T)$, then we may replace $Q$ with the shorter directed path $y_1\ldots y_{i-1}z$, contradicting the minimality of 
$|\bigcup \mQ|$. So we have $zy_{i-1} \in E(T)$. But then as long as $i \ge 3$ we may replace $P$ with the shorter path $zy_{i-1}x_j\ldots x_s$, 
contradicting the initial minimal choice of $|\bigcup \mP|$. It remains to consider when $i = 2$. In this case,
$zy_2 \notin E(T)$ for every $z \in L\setminus (L_A \cup L_B)$, since otherwise we can replace $P$ with a shorter 
directed path. Thus $y_2$ has at least $2k$ out-neighbours in $L$, and we can add $y_1y_2$ to the matching 
$\mM_A$, a contradiction to the maximality of this matching. It follows that $P \cap Q = \varnothing$ for $P \in \mP'$.

So let us assume that $P \in \mP''$. Since the paths in $\mQ$ avoid $\{x_1, x_2\}$, we may assume in this case 
that $y_i = x_j$ for some $j \ge 3$. The same argument as in the previous paragraph shows that we may assume $i\ge 3$ (otherwise, we obtain a larger matching
than 
$\mM_A$). Also, as before, if $z \in L \setminus (L_A \cup L_B)$, then $y_{i-1}z \notin E(T)$; otherwise we can replace 
$Q$ with the shorter path $y_1\ldots y_{i-1}z$. Hence $y_{i-1}$ has at least $|L| - |L_A \cup L_B| \ge 2k$ in-neighbours 
in $L$. Choose one of these in-neighbours $u$ (disjoint from $L_A \cup L_B$) and consider the path $P^* := uy_{i-1}x_j\ldots x_s$.
Then $P^*$ has the same length as $P$ and its second vertex has at least $2k$ in-neighbours in $L$, so we could replace 
$P$ with $P^*$, contradicting the maximality of $\mP'$. Therefore, we must have $P \cap Q = \varnothing$, and the proof of 
\Cref{claim:paths_disjoint} is complete.

\end{proof}

Armed with \Cref{claim:paths_disjoint}, the proof of \Cref{lem:reversing_paths} is essentially complete.
Indeed, every vertex in $W_A$ has at least $2k$ out-neighbours in $L$, and so each of these vertices has at least
\[
	2k - |L_A \cup L_B| = |W_A| + |W_B|,
\]
out-neighbours in $L\setminus (L_A \cup L_B)$. So for each vertex in $W_A$ we may select a distinct out-neighbour 
in $L \setminus (L_A \cup L_B)$. Then every vertex in $W_B$ has at least $|W_B|$ in-neighbours from the remaining 
vertices of $L$, so we can pick a distinct in-neighbour for every vertex of $W_B$. The paths of length $2$ using vertices of $W_A \cup W_B$ 
together with $\mP$ and $\mQ$ form the required collection of paths.
\end{proof}

We can now finish the proof of \Cref{thm:main}.

\begin{proof}[ of \Cref{thm:main}]

Let $k \ge 2$ be an integer and let $f(k) := d^*(12k^2) + 2k$, where $d^*: \mathbb{N} \rightarrow \mathbb{N}$ is the function provided by \Cref{cor:subdivision2}. Suppose that $T$ is a $4k$-connected tournament with minimum out-degree at least $f(k)$, and let $X = \{x_1, \ldots, x_k\}$, 
$Y = \{y_1, \ldots, y_k\}$ be two disjoint $k$-sets of vertices. We wish to find pairwise vertex disjoint directed
 paths going from $x_i$ to $y_i$ for each $i \in [k]$. Remove $X\cup Y$ from $T$; the tournament induced on $V(T) \setminus (X \cup Y)$ 
 has minimum out-degree at least $d^*(12k^2)$, so by \Cref{cor:subdivision2} we may embed in $T$ a $\mK_{12k^2}^*$ disjoint 
 from $X \cup Y$. Denote this subdivision by $\mS$. We shall use the same notation as in \Cref{sec:subdivision}, namely, $U$ denotes the 
 branch vertices of $\mS$, $\mK$ denotes the underlying minimal subdivision of $\overrightarrow{K}_{12k^2}$ composed of minimal paths $P_{uv}, P_{vu}$ 
 for every pair of branch vertices $u, v \in U$, and $\mL$ denotes the collection of minimal paths attached to $\mK$. A \emph{path of} $\mS$ refers to any path $P_{uv}$ between branch vertices of length at least $2$, and any member of $\mL$. Furthermore, we consider the following edges to belong to the structure $\mS$:

\begin{itemize}

\item The edges belonging to paths in $\mathcal{K}$, except the paths of length one.

\item The edges belonging to paths in $\mathcal{L}$.

\item For every pair $u, v \in U$, every edge in $T$ between $\{u, v\}$
	and $V(P_{uv})\cup V(P_{vu})$.

\item For every $u \in U$, every edge in $T$ between $u$ and $\bigcup \mathcal{L}_u$.

\end{itemize}

We denote the set of edges of $\mS$ by $E(\mS)$. For example, whenever we speak of distances in $\mS$, we insist that they are computed using only these directed edges. 
Let $\mP$ and $\mQ$ be any two collections of pairwise disjoint directed paths such that every path 
 in $\mP$ goes from $U$ to $Y$, every path in $\mQ$ goes from $X$ to $U$, and all of these paths are internally 
 vertex disjoint from $U$; by \Cref{lem:reversing_paths}, such collections exist.
 We say that a pair $(u, x) \in U \times V(\mS)$ is at \emph{in-distance} $d$ \emph{in} $\mS$ if $d$ is the smallest integer such that there is a directed path $P'$ of length $d$ using only edges of $\mS$, and 
 such that $P'$ goes from $u$ to $x$. We shall also sometimes say that $x$ has in-distance $d$ in $\mS$ from $u$. 
 Similarly, we say that $(u, x) \in U \times V(\mS)$ is at \emph{out-distance} $d$ \emph{in} $\mS$ if $d$ is the smallest integer such that 
  there is a directed path $Q'$ of length $d$ using only edges of $\mS$, and such that $Q'$ goes from $x$ to $u$ in $\mS$; we shall also sometimes say that 
  $x$ has out-distance $d$ in $\mS$ from $u$. We denote in-distance by $d^{\text{in}}(u, x)$ and out-distance by $d^{\text{out}}(u, x)$ (where we have 
  suppressed the dependence on $\mS$).

 \begin{obs}\label{obs:distances}
 Let $x \in V(\mS) \setminus U$. Then $x$ is at in-distance (or out-distance) at least $3$ from every vertex of $U$, except 
 possibly the branch vertex (or vertices) belonging to the path of $\mS$ containing $x$. 
  \end{obs}
  \begin{proof}
  If $x \in V(\mS) \setminus U$, then either $x \in P_{uv}$ for some $u, v \in U$ or $x \in L_{uv}^u \in \mL_u$ (or possibly both).
 Let $w \in U\setminus \{u, v\}$. In order to get from $w$ to $x$ using only edges of $\mS$, we must first reach either $u$ or $v$. However, recall that the single edge paths in $\mK$ are \emph{not} edges of $\mS$, so the path from $w$ to $u$ or $v$ in $\mS$ has length at least $2$. Therefore, 
 $x$ has in-distance at least $3$ from $w$, as required. A symmetric argument shows that the observation remains true with `out-distance' instead of `in-distance'.

  \end{proof}

 In the following, we shall always assume that any family $\mF$ of directed paths in $T$ between $X \cup Y$ and $U$ are internally disjoint from $U$. We also denote by $U_{\mF}$ the set $U \cap (\bigcup \mF)$. Our first claim asserts that we may assume the paths in one of the collections $\mP$, $\mQ$ contains few vertices which are `close' in $\mS$ to a vertex in $U$. To state it precisely, we say that a vertex $u \in U \setminus U_{\mP}$ is \emph{in-close} to a subset $S$ of vertices if $d^{\text{in}}(u, x) \le 2$ for some $x \in S$. Similarly, we say that $u \in U \setminus U_{\mQ}$ is \emph{out-close} to $S$ provided $d^{\text{out}}(u, x) \le 2$ for some $x \in S$.
 
 \begin{lem}\label{lem:branch}
We may choose either $\mP$ or $\mQ$ such that there are at most $8k^2+4k$ vertices in $U \setminus U_{\mP}$ (resp., $U \setminus U_{\mQ}$) that are in-close to $\bigcup \mP\setminus U_{\mP}$ (resp., out-close to $\bigcup \mQ\setminus U_{\mQ}$).
 \end{lem}

\begin{proof}
Apply \Cref{lem:reversing_paths} with $A = X$, $B = Y$, and $L = U$. Using the proof and notation of \Cref{lem:reversing_paths}, assume that $|W_X| \le |W_Y|$. Then recall that we may choose the paths from $U$ to $Y$ first minimally (with respect to the number of vertices used) upon the removal of $W_X \cup W_Y$, a set of at most $2k$ vertices. Recall also that each such path which uses a vertex of $W_X \cup W_Y$ has length two.
Suppose, to the contrary, that there is a set $U' \subset U \setminus U_{\mP}$ of more than $8k^2+4k$ vertices such that for every $u \in U'$ there is $x \in \bigcup\mP \setminus U_{\mP}$ with $d^{\text{in}}(u, x) \le 2$. We claim that this contradicts minimality.
Indeed, by the pigeonhole principle there is a set $U'_0 \subset U'$ of size more than $8k + 4$, and a path $P \in \mP$ such that for each 
$u \in U'_0$ there is some $x \in P$ with $d^{\text{in}}(u, x) \le 2$. From \Cref{obs:distances}, it follows that 
for each interior vertex $v$ of $P$ there are at most two vertices of $U'_0$ that are at in-distance $2$ from $v$. Therefore $P$ 
must have more than two edges so does not intersect $W_X \cup W_Y$. Note that $P$ contains at most one vertex at in-distance $1$ from a vertex in $U \setminus U_{\mP}$, as otherwise we may reroute $P$ and obtain a shorter path avoiding $W_X \cup W_Y$. Let $D$ denote the set of all vertices of $P$ at in-distance exactly $2$ from some vertex of $U'_0$, and consider the bipartite graph with vertex sets $U'_0$, $D$ where $uv$ is an edge whenever $d^{\text{in}}(u, v) = 2$. We claim that this bipartite graph has a matching of size at least $4k + 2$. Let $\mM$ be a maximum matching, and suppose $|\mM| < 4k + 2$. Let $M_1$, $M_2$ denote the endpoints of the matching in $U_0'$, $D$, respectively. Note that $|U'_0 \setminus M_1| > 8k + 4 - (4k + 2) = 4k + 2$. Also, the degree of every vertex in $U'_0 \setminus M_1$ is at least one, and there are no edges between $U'_0 \setminus M_1$ and $D \setminus M_2$ by the maximality of $\mM$. Since $|M_2| < 4k + 2$, it follows that some vertex in $M_2$ has degree at least $3$. But this contradicts \Cref{obs:distances}.

Therefore, we may choose a set $D' \subset D$ of at least $4k + 2$ vertices corresponding to distinct vertices of $U'_0$.
Let $P = p_0\ldots p_\ell$, where $p_0 \in U$ and $p_\ell \in X$, $F:= D' \setminus \{p_1, p_2\}$. For each $p_j \in F$, we may choose vertex disjoint  
directed paths $u_jm_jp_j$ of length $2$ in $\mS$, where $u_j \in U'_0$. Accordingly, there are at least $4k$ `middle vertices' $m_j$, at least 
$2k$ of which are disjoint from $W_X \cup W_Y$; let $M$ denote the set of middle vertices disjoint from $W_X \cup W_Y$. Now, suppose some $m_j \in M$ does not intersect any path in $\mP$. Then we may replace $P$ with the path $u_jm_jp_jP$, which is shorter and still avoids $W_X \cup W_Y$, a contradiction. Thus, each middle vertex in $M$ belongs to some member of $\mP$ and so by the pigeonhole principle there is a path $P'$ which contains at least two vertices of $M$. But both of 
these vertices are at in-distance $1$ from a vertex in $U\setminus U_{\mP}$, which, as noted before, is a contradiction. Hence at most $8k^2 + 4k$ vertices in $U \setminus U_{\mP}$ have the stated property, as claimed.
A symmetric argument shows that we may choose $\mQ$ with the stated property in the event that $|W_Y| \le |W_X|$. This completes the proof of the lemma.
\end{proof}

Suppose $\mF$ is a collection of pairwise disjoint directed paths from $U$ to $Y$ (internally disjoint from $U$), and let $P = p_0 \ldots p_t$ be any path in $\mF$.
We call the pairs $(p_0, p_1)$ and $(p_0, p_2)$ \emph{trivial} if they have in-distance at most $2$ in $\mS$; any other pair with in-distance at most $2$ is \emph{nontrivial}.
For a subset $U' \subseteq U$ we shall say that $\mF$ is $U'$-\emph{good} if no nontrivial pair of vertices from $U' \times \left(\bigcup\mF \setminus U_{\mF}\right)$ is at in-distance at most $2$ in $\mS$. In particular, each path $P \in \mF$ intersects $U'$ in at most one vertex, namely its initial vertex. 
Suppose that $\mF$ satisfies the property stated in \Cref{lem:branch}. Then we have the following:

\begin{claim}\label{claim:good}
There exists a subset $U' \subset U \setminus U_{\mF}$ of size at least $2k$ such that $\mF$ is $U'$-good.
\end{claim}

\begin{proof}
This follows immediately from the previous lemma. Indeed, remove from $U$ every vertex in $U_{\mF}$ and every vertex in $U \setminus U_{\mF}$
at in-distance at most $2$ in $\mS$ from some vertex of $\bigcup \mF \setminus U_{\mF}$; let $U'$ denote the remaining set of vertices. By \Cref{lem:branch}, we have removed at most $8k^2 + 5k$ vertices. As $|U| = 12k^2$ we have 
$|U'| \ge 12k^2 - (8k^2 + 5k) \ge 2k$, since $k \ge 2$. Clearly $\mF$ is $U'$-good.
\end{proof}

We shall assume without loss of generality that we may choose the paths from $U$ to $Y$ with the property stated in \Cref{lem:branch}. So
the previous two claims show that we may find collections of vertex disjoint directed paths $\mP, \mQ$ which are internally disjoint from $U$ and such that the paths in $\mP$ go from $U$ to $Y$, 
the paths in $\mQ$ go from $X$ to $U$, and $\mP$ is $U'$-good for some $U' \subset U \setminus U_{\mP}$ with $|U'| \ge 2k$.  
Conditioned on this, we assume that $\mP \cup \mQ$ minimizes the number of edges outside of $\mS$, and again conditioned on this, we take such a pair with $|\bigcup \mP| +| \bigcup \mQ|$ as small as possible.
Let $U'' = U' \setminus U_{\mQ}$ so that $|U''| \ge k$ and it is disjoint from $U_{\mP} \cup U_{\mQ}$; we may assume that $U'' = \{u_1, \ldots , u_k\}$ has precisely $k$ elements.
We now show that one can reroute the paths in $\mP \cup \mQ$ through 
$U''$ in order to create the desired paths linking $x_i$ to $y_i$ for each $i \in [k]$. Let $U_{\mP} = \{z_1, \ldots ,z_k\}$ and $U_{\mQ} = \{w_1, \ldots , w_k\}$ 
so that $z_i$ is the initial vertex in $U$ of the path $P_i \in \mP$ with terminal vertex $y_i \in Y$, and $w_i$ is the terminal vertex in $U$ of the path $Q_i \in \mQ$ with 
initial vertex $x_i \in X$. Recall that for every pair of branch vertices $u, v \in U$, $P_{uv}$ and $P_{vu}$ denotes the path in $\mK$ from $u$ to $v$, and from $v$ to $u$, respectively. The following sequence of claims show that we can control intersections of paths in $\mP \cup \mQ$ with appropriate paths 
in $\mS$ in order to link each $x_i$ to $y_i$. 

\begin{claim}\label{claim:intersect0}
Suppose some path $Q\in \mQ$ intersects $ L_{w_iu_i}^{u_i} \in \mL_{u_i}$, for some $i \in [k]$. Let $z$ be the first vertex of $L_{w_iu_i}^{u_i}$ in the intersection. Then one of the following holds:  $z$ is the terminal vertex of $ L_{w_iu_i}^{u_i}$ and $z\in Q_i$, or $z$ is the second vertex of $ L_{w_iu_i}^{u_i}$. 
\end{claim} 
\begin{proof}
Suppose $z$ is not the second vertex of $ L_{w_iu_i}^{u_i}$. If $z$ is an interior point of $ L_{w_iu_i}^{u_i}$, then $zu_i \in E(T)$ by minimality of the path $L_{w_iu_i}^{u_i}$. Note that if $Q$ has an edge which is not in $E(\mS)$ after $z$ then we have a contradiction: indeed replacing 
$Q$ with $Qzu_i$ yields a collection of paths with fewer edges outside of $E(\mS)$. Otherwise, $Q=Q_i$ and it must use at least $2$ edges after $z$, so we obtain a contradiction to the minimality of $|\bigcup \mP| + |\bigcup \mQ|$ by rerouting the path as before. 
Therefore, $z$ must be the terminal vertex of $L_{w_iu_i}^{u_i}$. Finally, $z$ must belong to $Q_i$, otherwise we may similarly reroute $Q$ through $u_i$, decreasing the number of edges used outside $E(\mS)$. 
\end{proof}

\begin{claim}\label{claim:intersect1}
No path in $\mP$ intersects $P_{w_iu_i}$. Moreover, if $q_i$ denotes the last vertex in $P_{w_iu_i}$ which occurs as the intersection of some path in $\mQ$, then $q_i \in Q_i$.
\end{claim}
\begin{proof}
No path in $\mP$ intersects $\{u_i, w_i\}$, so it suffices to show that no such path intersects the interior of $P_{w_iu_i}$.  
Therefore, we may assume that $P_{w_iu_i}$ has length at least $2$. Suppose first that some $P \in \mP$ contains a vertex $v$ in the interior. 
Note that $v$ must be the penultimate vertex of $P_{w_iu_i}$. Otherwise, $u_iv \in E(T) \cap E(\mS)$ by the minimality of the subdivision $\mK$, and 
this contradicts the fact that $\mP$ is $U'$-good. 
Consider the loop path $L = L_{w_iu_i}^{u_i} \in \mL_{u_i}$ at $u_i$ ending at $v$, and recall that 
the edges of $L$ are edges of $\mS$. Let $z$ be the first vertex in $L_{w_iu_i}^{u_i}$ belonging to some path $P' \in \mP$: such a 
vertex and path exist since we may take $z = v$ and $P' = P$. 
Let $L'$ be the initial segment of the path $L_{w_iu_i}^{u_i}$ ending at $z$.
 
Suppose first that no path in $Q\in \mQ$ intersects $L'$, and replace $P'$ with $P'' = u_iL'zP'$.
Since $P'$ cannot intersect $u_i$ or $w_i$ it must have an edge which is not in $E(\mS)$ before $z$. It follows that $P''$ has fewer edges outside of $\mS$. This is a contradiction to our choice of $\mP \cup \mQ$, provided $\mP'':= (\mP \setminus \{P'\}) \cup \{P''\}$ is $U'$-good. To see this, observe that any vertex of $L \setminus \{v\}$ is at in-distance at least $3$ from $w_i$. Moreover, if $w_i \in U'$, and $z=v$ (and hence $P' = P$), then $z$ is also at in-distance at least $3$ from $w_i$. Accordingly, if $w_i \in U'$, then every vertex of $P''$ is still at in-distance at least $3$ from $w_i$. By the minimality of $L$, every vertex in the interior of $L$ (except the second) is directed towards $u_i$; thus, the only vertices at in-distance at most $2$ from $u_i$ are the second and third vertices of $L$, say $x$ and $y$, respectively. But the pairs $(u_i, x)$ and $(u_i, y)$ are trivial pairs, and thus do not contradict $U'$-goodness. Lastly, by \Cref{obs:distances} every vertex of $P''$ (except possibly $u_i$) is at in-distance at least $3$ from every vertex of $U' \setminus \{u_i, w_i\}$. It follows that $\mP''$ is $U'$-good, which is a contradiction to our choice of $\mP \cup \mQ$. 

On the other hand, if some path $Q'\in \mQ$ intersects $L'$ in some vertex $r$, then by \Cref{claim:intersect0} $r$ must the second vertex of $L_{w_iu_i}^{u_i}$. Note that by $U'$-goodness, no path in $\mP$ contains the third vertex $r_1$ of $L_{w_iu_i}^{u_i}$, hence we can replace $Q'$ by $Q'rr_1u_i$ thus decreasing the number of edges outside $E(\mS)$.
Therefore we conclude that no path in $\mP$ can intersect $P_{w_iu_i}$.
Let us now show the second part of the claim. Suppose that $q_i \in Q_j$ for some $j \neq i$. Since $Q_j$ must avoid $\{u_i, w_i\}$ it contains 
an edge which is not in $E(\mS)$ after $q_i$. Replace $Q_j$ with $Q' = Q_jvP_{w_iu_i}$. Then by the previous paragraph, no path in $\mP$ intersects $Q'$ 
and the resulting collection of paths has fewer edges outside of $\mS$, a contradiction. This completes the proof of the claim.

\end{proof}

It remains to establish the analogous claims for the path $P_{u_iz_i}$, namely that intersections of paths in $\mP \cup \mQ$ with $P_{u_iz_i}$ 
and $L_{u_iz_i}^{u_i}$
behave as one expects. The arguments are similar to those in the previous two claims. \Cref{thm:main} will then be an immediate consequence.

\begin{claim}\label{claim:intersect00}
For every $i \in [k]$, no path in $\mP$ intersects $L_{u_iz_i}^{u_i} \in \mL_{u_i}$.
\end{claim}
\begin{proof}
Suppose some $P \in \mP$ intersects $L_{u_iz_i}^{u_i}$ in a vertex $z$. Then $z$ cannot be the first vertex of $L_{u_iz_i}^{u_i}$, 
as this would contradict the fact that $\mP$ is $U'$-good. 
Therefore, if $z'$ denotes the vertex preceding $z$ in $L_{u_iz_i}^{u_i}$, then by the minimality of paths in $\mL$, 
we have $u_iz' \in E(T) \cap E(\mS)$. But then $z$ is at in-distance $2$ from $u_i$, contradicting $U'$-goodness.

\end{proof}

\begin{claim}\label{claim:intersect2}
Let $p_i$ denote the first vertex in $P_{u_iz_i}$ which occurs as the intersection of some path in $\mP$. Then no path 
in $\mQ$ intersects $P_{u_iz_i}$ and $p_i \in P_i$.
\end{claim}

\begin{proof}
As before, it suffices to show that no path in $\mQ$ intersects the interior of $P_{u_iz_i}$, so we may assume that $P_{u_iz_i}$ has length 
at least $2$. Suppose some $Q \in \mQ$ intersects the interior of $P_{u_iz_i}$ at $v$. Note that since $Q$ does not meet $\{u_i, z_i\}$, it must leave $\mS$ at some time after $v$. If $v$ is not the second vertex of $P_{u_iz_i}$, then $vu_i \in E(T) \cap E(\mS)$, and so we may replace $Q$ with $Qvu_i$. This path 
has fewer edges outside of $\mS$ than $Q$, and this contradicts our minimal choice of $\mP \cup \mQ$. If $v$ is the second vertex, then let $L = L_{u_iz_i}^{u_i} \in \mL_{u_i}$ be the loop path at $u_i$ directed from $v$ to $u_i$. Let $z$ be the last vertex of $L$ which occurs as the intersection
of some path $Q' \in \mQ$ ($z$ and $Q'$ exist since we may take $z = v$ and $Q' = Q$), and let $L'$ be the subpath of $L$ from $z$ to $u_i$.
By \Cref{claim:intersect00}, no path in $\mP$ intersects $L'$, so replace $Q'$ with $Q'zL'u_i$. Again, the edges of $L'$ are in $E(\mS)$ so this path has fewer edges outside $\mS$ than $Q'$, a contradiction. It follows that no path in $\mQ$ intersects $P_{u_iz_i}$ as claimed.
For the second part of the claim, suppose that $p_i \in P_j$ for some $j \neq i$. Then $P_j$ avoids $\{u_i, z_i\}$ and therefore leaves $\mS$ at some time before $p_i$. 
Now, no path in $\mP \cup \mQ$ intersects the interior of the subpath $P_{u_iz_i}p_i$ so replace $P_j$ with $P' = P_{u_iz_i}p_iP_j$. 
This path has fewer edges outside of $\mS$. We claim that $\mP'  = (\mP \setminus \{P_j\}) \cup \{P'\}$ is $U'$-good. Indeed, note that since $\mP$ is $U'$-good, the subpath $P_{u_iz_i}p_i$ has length 
at least $3$. Also, for every $v \in P_{u_iz_i}$ we have that $vu_i \in E(T)$ by the minimality of $\mK$. So the only pairs at in-distance at most $2$ in 
$U' \times (\bigcup \mP' \setminus U_{\mP'})$ are the trivial pairs $(u_i, x)$ and $(u_i, y)$, where $x, y$ are the second and third vertices, respectively, of $P_{u_iz_i}$. But these pairs, by definition, do not contradict $U'$-goodness. It follows that $j = i$, and the claim is proved.

\end{proof}

By Claims~\ref{claim:intersect1} and~\ref{claim:intersect2}, the directed paths $Q_iq_iP_{w_iu_i}u_iP_{u_iw_i}p_iP_i$, for each $i \in [k]$, 
are pairwise vertex disjoint and link $x_i$ to $y_i$. This completes the proof of \Cref{thm:main}.

\end{proof}

\section{Final remarks and open problems}\label{sec:final}

The most obvious open problem is to reduce our bound of $4k$ on the connectivity in \Cref{thm:main}.  
We remark that an improvement on the connectivity bound in \Cref{lem:reversing_paths} translates directly into a better bound in \Cref{thm:main}. Unfortunately, we could not go beyond $4k$. Furthermore, \Cref{lem:reversing_paths} does \emph{not} hold if 
we replace $4k$ with anything smaller than $3k$. The following 
construction, of a $(3k-1)$-connected tournament $T$ where \Cref{lem:reversing_paths} fails, was communicated to us by Kamil Popielarz. Suppose $V(T) = [n]$ and partition $V(T)$ into disjoint sets $A, S, B, L$, where $L = V(T) \setminus (A \cup S \cup B)$, and $|A| = |B| = k$, $|S| = 2k - 1$. Direct the edges from $L$ to $A$; from $B$ to $L$; from $A$ to $S$ and from $S$ to $B$; and from $A$ to $B$. Inside $L$ we place a balanced blow-up of a directed triangle. That is, equitably partition $L$ 
into sets $L_1, L_2, L_3$ with directed edges $L_1 \rightarrow L_2$, $L_2 \rightarrow L_3$, $L_3 \rightarrow L_1$, and inside each of the $L_i$'s we orient the edges arbitrarily. Now, join every vertex in $S$ to all of $L_1$ and join every vertex of $L_2$ to all of $S$. Finally, orient the edges between $S$ and $L_3$, and the edges inside $A, B$, and $S$, arbitrarily. 

Provided $n$ is sufficiently large (depending on $k$), it is not hard to show that 
$T$ is $(3k-1)$-connected. Observe that we cannot get from $A$ to $L$ (disjointly from $B$) without using vertices of $S$. Similarly, we cannot get from $L$ to $B$ (disjointly from $A$) without using vertices of $S$. As $|S| = 2k - 1$, any path system as in \Cref{lem:reversing_paths} will not be pairwise disjoint. Accordingly, \Cref{lem:reversing_paths} fails for this tournament. We remark that a slight modification of this 
construction yields a tournament which additionally has large minimum in and out-degree.

Aside from improving our bound of $4k$ on the connectivity and resolving completely Pokrovskiy's conjecture, there are a few other 
open problems of interest. For example, what is the smallest function $d(k)$ such that \Cref{thm:subdivision} holds?

\begin{prob}
Determine the smallest function $d: \mathbb{N} \rightarrow \mathbb{N}$ such that any tournament 
$T$ with $\delta^+(T) \ge d(k)$ contains a subdivision of the complete directed graph $\overrightarrow{K}_k$.
\end{prob}

Note that our proof gives a doubly exponential bound on $d(k)$. Indeed, it is easy to check that $d(k)\leq 2^{2^{Ck^2}}$. 
Finally, while the conclusion of \Cref{thm:subdivision} does not hold if we replace $T$ with a general digraph, can we embed subdivisions of \emph{acyclic} digraphs
in digraphs of large minimum out-degree?
We end by recalling the following beautiful conjecture of Mader~\cite{Mader3} from 1985.

\begin{conj}
 For every positive integer $k$, there exists a function $f(k)$ such that every digraph with minimum out-degree at least $f(k)$ contains a subdivision of the transitive tournament of order $k$.
\end{conj}

Of course, since every acyclic digraph is contained in the transitive tournament of the same order, this conjecture (if true) would give an affirmative answer to the preceding 
question.

\section{Acknowledgements}
We would like to express our thanks to Kamil Popielarz for very helpful discussions throughout this project. We would also like to thank the referees for helpful comments.

    \bibliography{linked_tourn}
    \bibliographystyle{amsplain}
\end{document}